\documentclass[12pt, a4paper]{amsart}
\usepackage{amsmath,amsfonts,amssymb,amsthm, color, mathrsfs}
\usepackage[all]{xy}
\usepackage{hyperref}
\usepackage{bm}
\usepackage{relsize}
\allowdisplaybreaks

\begin{document}
\newtheorem{thm}{Theorem}[section]
\newtheorem{prop}[thm]{Proposition}
\newtheorem{coro}[thm]{Corollary}
\newtheorem{conj}[thm]{Conjecture}
\newtheorem{example}[thm]{Example}
\newtheorem{lem}[thm]{Lemma}
\newtheorem{rem}[thm]{Remark}
\newtheorem{convention}[thm]{Convention}
\newtheorem{hy}[thm]{Hypothesis}
\newtheorem*{acks}{Acknowledgements}
\theoremstyle{definition}
\newtheorem{de}[thm]{Definition}
\newtheorem{bfproof}[thm]{{\bf Proof}}

\newcommand{\C}{{\mathbb{C}}}
\newcommand{\Z}{{\mathbb{Z}}}
\newcommand{\N}{{\mathbb{N}}}
\newcommand{\Q}{{\mathbb{Q}}}
\newcommand{\te}[1]{\textnormal{{#1}}}
\newcommand{\set}[2]{{
    \left.\left\{
        {#1}
    \,\right|\,
        {#2}
    \right\}
}}
\newcommand{\sett}[2]{{
    \left\{
        {#1}
    \,\left|\,
        {#2}
    \right\}\right.
}}
\def \<{{\langle}}
\def \>{{\rangle}}

\def \({ \left( }
\def \){ \right) }
\def \:{\mathopen{\overset{\circ}{
    \mathsmaller{\mathsmaller{\circ}}}
    }}
\def \;{\mathclose{\overset{\circ}{\mathsmaller{\mathsmaller{\circ}}}}}

\newcommand{\g}{{\frak g}}
\newcommand{\fg}{\g}
\newcommand{\fk}{{\frak k}}
\newcommand{\gc}{{\bar{\g}}}
\newcommand{\h}{{\frak h}}
\newcommand{\n}{{\frak n}}
\newcommand{\borel}{{\frak b}}
\newcommand{\cent}{{\frak c}}
\newcommand{\notc}{{\not c}}
\newcommand{\Loop}{{\cal L}}
\newcommand{\G}{{\cal G}}
\newcommand{\al}{\alpha}
\newcommand{\alck}{\al^\vee}
\newcommand{\be}{\beta}
\newcommand{\beck}{\be^\vee}
\newcommand{\ssl}{{\mathfrak{sl}}}
\newcommand{\rk}{{\mathrm{k}}}
\newcommand{\rd}{{\mathrm{d}}}
\newcommand{\ft}{{\mathfrak{t}}}
\newcommand{\rc}{{\mathrm{c}}}
\newcommand{\CE}{{\mathcal{E}}}
\newcommand{\CH}{{\mathcal{H}}}

\newcommand{\chg}{\check{\mathfrak g}}
\newcommand{\chh}{\check{\mathfrak h}}
\newcommand{\cha}{\check{a}}
\newcommand{\chb}{\check{b}}
\newcommand{\ka}{\mathfrak{a}}
\newcommand{\chka}{\check{\mathfrak{a}}}
\newcommand{\CL}{\mathcal L}

\newcommand{\rtu}{{\xi}}
\newcommand{\period}{{N}}
\newcommand{\half}{{\frac{1}{2}}}
\newcommand{\quar}{{\frac{1}{4}}}
\newcommand{\oct}{{\frac{1}{8}}}
\newcommand{\hex}{{\frac{1}{16}}}
\newcommand{\reciprocal}[1]{{\frac{1}{#1}}}
\newcommand{\inverse}{^{-1}}
\newcommand{\inv}{\inverse}
\newcommand{\SumInZm}[2]{\sum\limits_{{#1}\in\Z_{#2}}}
\newcommand{\uce}{{\mathfrak{uce}}}
\newcommand{\choice}[2]{{
\left[
\begin{array}{c}
{#1}\\{#2}
\end{array}
\right]
}}

\newcommand{\wh}[1]{\widehat{#1}}

\newlength{\dhatheight}
\newcommand{\dwidehat}[1]{%
    \settoheight{\dhatheight}{\ensuremath{\widehat{#1}}}%
    \addtolength{\dhatheight}{-0.35ex}%
    \widehat{\vphantom{\rule{1pt}{\dhatheight}}%
    \smash{\widehat{#1}}}}
\newcommand{\dhat}[1]{%
    \settoheight{\dhatheight}{\ensuremath{\hat{#1}}}%
    \addtolength{\dhatheight}{-0.35ex}%
    \hat{\vphantom{\rule{1pt}{\dhatheight}}%
    \smash{\hat{#1}}}}

\newcommand{\dwh}[1]{\dwidehat{#1}}

\newcommand{\dis}{\displaystyle}

\newcommand{\ot}{\otimes}

\newcommand{\overit}[2]{{
    \mathop{{#1}}\limits^{{#2}}
}}
\newcommand{\belowit}[2]{{
    \mathop{{#1}}\limits_{{#2}}
}}
\xymatrixcolsep{5pc}

\newcommand{\orb}[1]{|\mathcal{O}({#1})|}
\newcommand{\up}{_{(p)}}
\newcommand{\uq}{_{(q)}}
\newcommand{\upq}{_{(p+q)}}
\newcommand{\uz}{_{(0)}}
\newcommand{\uk}{_{(k)}}
\newcommand{\nsum}{\SumInZm{n}{\period}}
\newcommand{\ksum}{\SumInZm{k}{\period}}
\newcommand{\overN}{\reciprocal{\period}}
\newcommand{\df}{\delta\left( \frac{\xi^{k}w}{z} \right)}
\newcommand{\dfl}{\delta\left( \frac{\xi^{\ell}w}{z} \right)}
\newcommand{\ddf}{\left(D\delta\right)\left( \frac{\xi^{k}w}{z} \right)}

\newcommand{\ldfn}[1]{{\left( \frac{1+\xi^{#1}w/z}{1-{\xi^{#1}w}/{z}} \right)}}
\newcommand{\rdfn}[1]{{\left( \frac{{\xi^{#1}w}/{z}+1}{{\xi^{#1}w}/{z}-1} \right)}}
\newcommand{\ldf}{{\ldfn{k}}}
\newcommand{\rdf}{{\rdfn{k}}}
\newcommand{\ldfl}{{\ldfn{\ell}}}
\newcommand{\rdfl}{{\rdfn{\ell}}}

\newcommand{\kprod}{{\prod\limits_{k\in\Z_N}}}
\newcommand{\lprod}{{\prod\limits_{\ell\in\Z_N}}}
\newcommand{\E}{{\mathcal{E}}}
\newcommand{\F}{{\mathcal{F}}}
\newcommand{\fsl}{{\mathfrak{sl}}}

\newcommand{\wt}[1]{\widetilde{#1}}

\newcommand{\tar}{{\uce\left(\Loop\left(\gc\right)\right)}}
\newcommand{\U}{{\mathcal{U}}}
\newcommand{\qtar}{\U_\hbar(\wh{\g})}
\newcommand{\qtartemp}[1]{\U_{q^{#1}}(\wh{\g})}
\newcommand{\qptar}{\wh{\U}_\hbar(\wh{\g})}
\newcommand{\qhei}{{\U_\hbar( \wh{\h} )}}
\newcommand{\htart}[1]{{\mathcal{#1}}_\hbar(\wh\fg)}
\newcommand{\hheit}[1]{{\mathcal{#1}}_\hbar(\wh\h)}
\newcommand{\qptart}[1]{\wh{\mathcal{#1}}_\hbar(\wh\fg)}
\newcommand{\qphei}{{\wh\U_\hbar(\wh\h)}}
\newcommand{\qpheit}[1]{\wh{\mathcal{#1}}_\hbar(\hat\h)}
\newcommand{\symalg}{{\mathcal S}}

\newcommand{\qheip}{{\U_\hbar( \wh{\h}^+ )}}
\newcommand{\qhein}{{\U_\hbar( \wh{\h}^- )}}

\newcommand{\comp}{{\mathfrak{comp}}}

\newcommand{\ctimes}{{\widehat{\otimes}}}

\newcommand{\bigctimes}{{\ \widehat{\bigotimes}\ }}

\newcommand{\hctvs}[1]{Hausdorff complete linear topological vector space}

\newcommand{\act}{\triangleright}

\newcommand{\T}{\mathcal{T}}

\newcommand{\varprodright}{\mathop{\underrightarrow{\prod}}\limits}

\numberwithin{equation}{section}

\title{On quantum toroidal algebra of type $A_1$}

\author{Fulin Chen$^1$}
\address{School of Mathematical Sciences, Xiamen University,
 Xiamen, China 361005} \email{chenf@xmu.edu.cn}\thanks{$^1$Partially supported by the Fundamental
Research Funds for the Central Universities (No.20720190069) and NSF of China (No.11971397).}

\author{Naihuan Jing$^2$}
\address{Department of Mathematics, North Carolina State University, Raleigh, NC 27695,
USA}
\email{jing@math.ncsu.edu}
\thanks{$^2$Partially supported by NSF of China (No.11531004, No.11726016) and Simons Foundation (No.523868).}

\author{Fei Kong$^3$}
\address{Key Laboratory of Computing and Stochastic Mathematics (Ministry of Education), School of Mathematics and Statistics, Hunan Normal University, Changsha, China 410081} \email{kongmath@hunnu.edu.cn}
\thanks{$^3$Partially supported by NSF of China (No.11701183).}

 \author{Shaobin Tan$^4$}
 \address{School of Mathematical Sciences, Xiamen University,
 Xiamen, China 361005} \email{tans@xmu.edu.cn}
 \thanks{$^4$Partially supported by NSF of China  (Nos.11531004, 11971397).}


\begin{abstract} In this paper we introduce a new quantum algebra which  specializes to the $2$-toroidal Lie algebra of type $A_1$.
 We  prove that this quantum toroidal algebra
has a natural triangular decomposition,  a (topological) Hopf algebra structure and a vertex operator realization.
\end{abstract}
\maketitle

\section{Introduction}
Let $\dot\g$  be a finite dimensional simple Lie algebra over $\C$.
The  universal central extension $\mathfrak t(\dot\g)$ of the $2$-loop algebra
$\dot\g\ot \C[t_1^{\pm 1}, t_2^{\pm 1}]$, called the toroidal Lie algebra, has a celebrated presentation given 
by Moody-Rao-Yokonuma \cite{MRY} for constructing
the vertex representation for $\mathfrak t(\dot\g)$. 
In understanding the Langlands reciprocity for algebraic surfaces,
Ginzburg-Kapranov-Vasserot \cite{GKV} introduced a notion of quantum toroidal algebra $\U_\hbar(\dot{\g}_{tor})$ associated to
 $\dot\g$.
The algebra $\U_\hbar(\dot{\g}_{tor})$ specializes to the Moody-Rao-Yokonuma presentation of $\mathfrak t(\dot\g)$ in general,
except for $\dot\g$ in type $A_1$ when $\U_\hbar(\dot{\g}_{tor})$ specializes to a {\it proper} quotient of the latter
\cite{E-PBW-qaff}. 
The theory of quantum toroidal algebras has been extensively studied, especially with a rich representation theory developed by
Hernandez \cite{He-representation-coprod-proof,He-drinfeld-coproduct} and others, see
\cite{He-total} for a survey.
One notices that two major structural properties of $\U_\hbar(\dot{\g}_{tor})$ have played a fundamental role in Hernandez's work:
the triangular decomposition and the (deformed) Drinfeld coproduct.

Let $A$ be the generalized Cartan matrix associated to the affine Lie algebra $\g$ of $\dot\g$.
When $A$ is symmetric, by using the vertex operators calculus, Jing introduced in \cite{J-KM} a quanutm affinization algebra $\U_\hbar(\hat\g)$ associated to $\g$.
Meanwhile, it is remarkable that finite dimensional representations of $\U_\hbar(\hat\g)$ were studied by Nakajima in \cite{Naka-quiver} using powerful geometric approach of quiver varieties.
If $A$ is of simply-laced type, then $\U_\hbar(\hat\g)$ is nothing but the quantum toroidal algebra $\U_\hbar(\dot{\g}_{tor})$.
However, for the case that $A$ is not of simply-laced type, the definition of $\U_\hbar(\hat\g)$
is slightly different from that of $\U_\hbar(\dot{\g}_{tor})$. Explicitly, one notices that $A_1^{(1)}$ is the unique symmetric but non-simply-laced affine generalized Cartan matrix.
In this case, the defining currents $x_0^\pm(z), x_1^\pm(z)$  in $\U_\hbar(\hat\g)$ satisfy the relation
\begin{align}\label{intro-re1}
(z-q^{\mp 2}w)(z-w) x_0^\pm(z)x_1^\pm(w)=(q^{\mp 2}z-w)(z-w) x_1^\pm(w)x_0^\pm(z),
\end{align}
which appeared naturally in calculations of  quantum vertex operators \cite{J-KM} and equivariant K-homology of quiver varieties \cite{Naka-quiver}.
In particular,   $\U_\hbar(\hat\g)$ specializes to the toroidal Lie algebra $\mathfrak t(\dot\g)$  of type $A_1$ as
the classical limit of \eqref{intro-re1} holds in $\mathfrak t(\dot\g)$.
On the other hand, in $\U_\hbar(\dot{\g}_{tor})$ these two currents satisfy the relation
\begin{align}\label{intro-re2}
(z-q^{\mp 2}w) x_0^\pm(z)x_1^\pm(w)=(q^{\mp 2}z-w) x_1^\pm(w)x_0^\pm(z).
\end{align}
This stronger relation was needed in verifying the compatibility with affine quantum Serre relations in $\U_\hbar(\dot{\g}_{tor})$
so that it processes a canonical triangular decomposition \cite{He-representation-coprod-proof}.
For the algebra $\U_\hbar(\hat\g)$, we only know that it has a weak form of triangular decomposition \cite{Naka-quiver}.

From now on, we assume that $\dot\g$ is of type $A_1$.
The main goal of this paper is to define a ``middle" quantum algebra
\[\U_\hbar(\hat\g)\twoheadrightarrow\U\twoheadrightarrow \U_\hbar(\dot\g_{tor})\] of $\U_\hbar(\hat\g)$  and $\U_\hbar(\dot\g_{tor})$,
and prove that this new quantum toroidal algebra $\U$ processes the ``good'' properties
enjoyed by
both of  $\U_\hbar(\hat\g)$  and $\U_\hbar(\dot\g_{tor})$.
Explicitly, we first introduce in Section 2 a quantum algebra $\U$ which specializes to the toroidal Lie algebra $\mathfrak t(\dot\g)$.
By definition, $\U$ is the quotient algebra of  $\U_\hbar(\hat\g)$ obtained by modulo the relation
\begin{align}\label{intro-re3}
[x_0^\pm(z_1),(z_2-q^{\mp 2}w) x_0^\pm(z_2)x_1^\pm(w)-(q^{\mp 2}z_2-w) x_1^\pm(w)x_0^\pm(z_2)]=0.
\end{align}
One notices that $\U_\hbar(\dot\g_{tor})$ is a quotient algebra of $\U$ as the relation \eqref{intro-re2} implies the relations \eqref{intro-re1} and \eqref{intro-re3}.
In Section 3, we prove that $\U$ admits a triangular decomposition (see Theorem \ref{thm:tri-decomp}).
In Section 4, we prove that $\U$ has  a deformed Drinfeld coproduct (see Theorem \ref{thm:bialg-struct}).
As in \cite{Gr-qshuff-qaff}, this allows us to define a (topological) Hopf algebra structure on $\U$ (see Theorem \ref{thm:hopf}).
As usual, the crucial step in establishing Theorems \ref{thm:tri-decomp} and \ref{thm:bialg-struct} is to check the compatibility with affine quantum Serre relations,
in where the new relation \eqref{intro-re3} appeared naturally (see \eqref{newrela1} and \eqref{newrela2}).
Finally, in Section 5 we point out that the quantum vertex operators constructed in \cite{J-KM} satisfy the relation \eqref{intro-re3}, and so 
 we obtain a  vertex representation for $\U$.

Throughout this paper, we denote by $\C[[\hbar]]$ the ring of complex formal series in one variable $\hbar$.
By a $\C[[\hbar]]$-algebra, we mean a  topological algebra over $\C[[\hbar]]$ with respect to the $\hbar$-adic topology.
For $n,k,s\in\Z$ with $0\le k\le s$, we denote the usual quantum numbers as follows
\begin{align*}
  [n]_q=\frac{q^n-q^{-n}}{q-q\inv},\quad[s]_q!=[s]_q[s-1]_q\cdots[1]_q,\quad
  \binom{s}{k}_q=\frac{[s]_q!}{[k]_q![s-k]_q!},
  \end{align*}
  where
\[q=\exp(\hbar)\in\C[[\hbar]].\]

\section{Quantum toroidal algebra of type $A_1$.}
In this section we introduce a new quantum algebra which specializes to toroidal Lie algebra of type $A_1$.

Let
\[A=(a_{ij})_{i,j=0,1}=\left(
                         \begin{array}{cc}
                           2 & -2 \\
                          -2 & 2 \\
                         \end{array}
                       \right)\]
 be the generalized Cartan matrix  of type $A_1^{(1)}$.
For $i,j=0,1$, let 
\begin{align}
  &g_{ij}(z)=\frac{q^{a_{ij}}-z}{1-q^{a_{ij}}z}
\end{align}
be the formal Taylor series at $z=0$. 
The following is  the main object of this paper:

\begin{de}\label{de:q-aff}
The quantum toroidal algebra $\U$  is the $\C[[\hbar]]$-algebra topologically generated by the elements
\begin{align}\label{eq:gen-set}
  h_{i,n},\ x_{i,n}^\pm,\ c\quad i=0,1,\ n\in\Z,
\end{align}
and subject to the relations in terms of generating functions in $z$:
\begin{align*}
  \phi_i^\pm(z)=q^{\pm h_{i,0}}\exp\(\pm(q-q\inv)\sum_{\pm n>0}h_{i,n}z^{-n}\),\quad
  x_i^\pm(z)=\sum_{n\in\Z}x_{i,n}^\pm z^{-n}.
\end{align*}
The relations are:
\begin{align*}
  &\te{(Q1)}\quad c\ \te{is central},\quad [\phi_i^\pm(z),\phi_j^\pm(w)]=0,\\
  &\te{(Q2)}\quad \phi_i^+(z)\phi_j^-(w)=\phi_j^-(w)\phi_i^+(z)g_{ij}(q^cw/z)\inv g_{ij}(q^{-c}w/z),\\
 &\te{(Q3)}\quad \phi_i^+(z)x_j^\pm(w)=x_j^\pm(w)\phi_i^+(z)g_{ij}(q^{\mp\half c}w/z)^{\pm 1},\\
  &\te{(Q4)}\quad \phi_i^-(z)x_j^\pm(w)=x_j^\pm(w)\phi_i^-(z)g_{ji}(q^{\mp\half c}z/w)^{\mp 1},\\
  &\te{(Q5)}\quad [x_i^+(z),x_j^-(w)]=\frac{\delta_{ij}}{q-q\inv}
  \(\phi_i^+(zq^{-\half c})\delta\(\frac{q^cw}{z}\)
  -\phi_i^-(zq^{\half c})\delta\(\frac{q^{-c}w}{z}\)\),\\
  &\te{(Q6)}\quad (z-q^{\pm 2}w)x_i^\pm(z)x_i^\pm(w)=
  (q^{\pm 2}z-w)x_i^\pm(w)x_i^\pm(z),\\
  &\te{(Q7)}\quad (z-q^{\mp 2}w)(z-w)x_i^\pm(z)x_j^\pm(w)
  =(q^{\mp 2}z-w)(z-w)x_j^\pm(w)x_i^\pm(z),\\
  &\te{(Q8)}\quad [x_i^\pm(z_1),\((z_2-q^{\mp 2}w)x_i^\pm(z_2)x_j^\pm(w)
    -(q^{\mp 2}z-w)x_j^\pm(w)x_i^\pm(z_2)\)]=0,\\
  &\te{(Q9)}\, \sum_{\sigma\in S_3}\sum_{r=0}^3(-1)^r\binom{3}{r}_q
   x_i^\pm(z_{\sigma(1)})\cdots x_i^\pm(z_{\sigma(r)})x_j^\pm(w)
   x_i^\pm(z_{\sigma(r+1)})\cdots x_i^\pm(z_{\sigma(3)})=0,
\end{align*}
where $i,j=0,1$ with $i\ne j$ in (Q7), (Q8), (Q9) and $\delta(z)=\sum_{n\in\Z}z^n$ is the usual $\delta$-function.
\end{de}

\begin{rem}{\em As indicated in Introduction,
in literature there are two other definitions of quantum toroidal algebra of type $A_1$:
the algebra $\U_\hbar(\hat\g)$ introduced in \cite{J-KM,Naka-quiver} and the algebra $\U_\hbar(\dot\g_{tor})$ introduced in \cite{GKV,He-representation-coprod-proof}.
By definition, the algebra $\U_\hbar(\hat\g)$ is
the $\C[[\hbar]]$-algebra topologically generated by the elements as in \eqref{eq:gen-set} with relations (Q1)-(Q7) and (Q9),
while the algebra $\U_\hbar(\dot\g_{tor})$ is
the $\C[[\hbar]]$-algebra topologically generated by the elements as in \eqref{eq:gen-set} with relations (Q1)-(Q6), (Q9) and the following relation
\begin{align}\label{strongerrela}
(z-q^{\mp 2}w)x_i^\pm(z)x_j^\pm(w)
  =(q^{\mp 2}z-w)x_j^\pm(w)x_i^\pm(z),\quad i\ne j\in \{0,1\}.
\end{align}
By definition we have that  $\U$  is a  quotient algebra of $\U_\hbar(\hat\g)$, while $\U_\hbar(\dot\g_{tor})$ is a quotient algebra of $\U$.
}\end{rem}

Now we recall the definition of the toroidal Lie algebra of type $A_1$.
Let $\mathcal K$ be the $\C$-vector space spanned by the symbols
\begin{align*}
  &t_1^{m_1}t_2^{m_2}\rk_i,\quad i=1,2,\ m_1,m_2\in\Z
\end{align*}
subject to the relations
\begin{align*}
  &m_1t_1^{m_1}t_2^{m_2}\rk_1+m_2t_1^{m_1}t_2^{m_2}\rk_2=0.
\end{align*}

Let $\dot\g=\ssl_2(\C)$ be the simple Lie algebra of type $A_1$ and $\<\cdot,\cdot\>$ the Killing form on $\dot\g$.
The toroidal Lie algebra (see \cite{MRY})
\begin{align*}
\mathfrak t=\mathfrak t(\dot\g)=\(\dot\g\ot\C[t_1^{\pm 1},t_2^{\pm 1}]\)\oplus \mathcal K
\end{align*}
is the universal central extension of the double loop algebra $\dot\g\ot\C[t_1^{\pm 1},t_2^{\pm 1}]$, where $\mathcal K$ is the center space and
\begin{align*}
  [x\ot t_1^{m_1}t_2^{m_2},y\ot t_1^{n_1}t_2^{n_2}]
  =[x,y]\ot t_1^{m_1+n_1}t_2^{m_2+n_2}+\<x,y\>(\sum_{i=1}^2 m_it_1^{m_1+n_1}t_2^{m_2+n_2}\rk_i),
\end{align*}
for $x,y\in\dot\g$ and $m_1,m_2,n_1,n_2\in\Z$.

Let $\{e^+,\al,e^-\}$ be a standard $\ssl_2$-triple in $\dot\g$, that is,
\begin{align*}
  &[e^+,e^-]=\al,\quad [\al,e^\pm]=\pm 2e^\pm.
\end{align*}
For $i=0,1$ and $m\in \Z$, set
\begin{align*}
  \al_{1,m}=\al\ot t_2^m,\ \al_{0,m}=t_2^mk_1-\al\ot t_2^m,\
  e_{1,m}^\pm=e^\pm\ot t_2^m,\ e_{0,m}^\pm=e^\mp\ot t_1^{\pm 1}t_2^m.
\end{align*}
Note that these elements generate the algebra $\mathfrak t$.

Following \cite{MRY}, we have:
\begin{prop}\label{MRY-pre} The toroidal Lie algebra $\mathfrak t$ is abstractly generated by
the elements $\al_{i,m}, e_{i,m}^\pm, \rk_1$ for $i=0,1, m\in \Z$ with relations
\begin{align*}
  &\te{(L1)}\quad [\rk_1,\mathfrak t]=0,\ [\al_{i,m},\al_{j,n}]=a_{ij}\delta_{m+n,0}m \rk_1,\\
  &\te{(L2)}\quad [\al_{i,m},e_{j,n}^\pm]=\pm a_{ij} e_{j,m+n},\\
  &\te{(L3)}\quad [e_{i,m}^+,e_{j,n}^-]=\delta_{ij}\(\al_{j,m+n}+m\delta_{m+n,0}\rk_1\),\\
  &\te{(L4)}\quad (z-w)[e_i^\pm(z),e_i^\pm(w)]=0,\\
  &\te{(L5)}\quad
  (z-w)^2[e_i^\pm(z),e_j^\pm(w)]=0,\quad i\ne j,\\
  &\te{(L6)}\quad (z_2-w)[e_i^\pm(z_1),[e_i^\pm(z_2),e_j^\pm(w)]]=0,\quad i\ne j,\\
  &\te{(L7)}\quad [e_i^\pm(z_1),[e_i^\pm(z_2),[e_i^\pm(z_3),e_j^\pm(w)]]]=0,\quad i\ne j,
\end{align*}
where $i,j=0,1$, $m,n\in \Z$ and $e_i^\pm(z)=\sum_{n\in\Z}e_{i,n}^\pm z^{-n}$.
\end{prop}
\begin{proof} Denote by $\mathcal L$ the Lie algebra abstractly generated by
the elements $\al_{i,m}, e_{i,m}^\pm, \rk_1$ for $i=0,1, m\in \Z$ with relations
(L1)-(L7). One easily checks that the relations (L1)-(L7) hold in $\mathfrak t$ and so we have a
surjective Lie homomorphism $\psi$ from $\mathcal L$ to $\mathfrak t$.  On the other hand,
denote by $\mathcal L'$ the Lie algebra abstractly generated by
the elements $\al_{i,m}, e_{i,m}^\pm, \rk_1$ for $i=0,1, m\in \Z$ with relations
(L1)-(L4) and (L7). Then there is a quotient map, say $\varphi$, from $\mathcal L'$ to $\mathcal L$.
 It was proved in \cite{MRY}  that
the surjective homomorphism $\psi\circ \varphi:\mathcal L'\rightarrow \mathcal L\rightarrow \mathfrak t$ is an isomorphism,
noting that the relation (L4) is equivalent to the relation $[e_i^\pm(z),e_i^\pm(w)]=0$.
This in turn implies that the map $\psi$ is an isomorphism, as required.
\end{proof}

By combing Definition \ref{de:q-aff} with Proposition \ref{MRY-pre}, one immediate gets the following result.

\begin{thm}
The classical limit $\U/\hbar\U$ of $\U$ is isomorphic to the universal enveloping algebra $\U(\mathfrak t)$ of the torodial Lie algebra $\mathfrak t$.
\end{thm}

\begin{rem}{\em From the proof of Proposition \ref{MRY-pre}, one knows that the algebra $\U_\hbar(\hat\g)$ also specializes $\mathfrak t$.
On the other hand, it is straightforward to see that the current
\[(z-w)[e_0^\pm(z),e_1^\pm(w)]\]  is nonzero in $\mathfrak t$  and its components  lie in the space $\bar{\mathcal K}=\sum_{m_1\in \Z}(\C t_1^{m_1}t_2\rk_1+\C t_1^{m_1}t_2^{-1}\rk_1)$.
 Thus, the algebra $\U(\dot{\g}_{tor})$ specializes to the
quotient algebra
$\mathfrak t/\bar{\mathcal K}$ of $\mathfrak t$ (cf. \cite{E-PBW-qaff}).}
\end{rem}

\section{Triangular decomposition of $\U$}\label{sec:tri-decomp}
In this section, we prove a triangular decomposition of $\U$.
By  a triangular decomposition of a $\C[[\hbar]]$-algebra $A$, we mean  a data of three closed $\C[[\hbar]]$-subalgebras $(A^-,H,A^+)$
of $A$ such that the multiplication $x^-\ot h\ot x^+\mapsto x^-hx^+$
induces an $\C[[\hbar]]$-module isomorphism from $A^-\wh\ot H\wh\ot A^+$ to $A$.
Here and henceforth, for two $\C[[\hbar]]$-modules $U,V$, the notion $U\wh\ot V$ stands for the $\hbar$-adically completed tensor product of $U$ and $V$.

Let $\U^+$ (resp. $\U^-$; resp. $\mathcal H$) be the closed subalgebra of $\U$ generated by $x_{i,m}^+$ (resp. $x_{i,m}^-$; resp. $h_{i,m}$, $c$).
The following is the main result of this section:

\begin{thm}\label{thm:tri-decomp}
$(\U^-,\mathcal H,\U^+)$ is a triangular decomposition of $\U$.
Moreover, $\U^+$ (resp. $\U^-$; resp $\mathcal H$) is isomorphic to the $\C[[\hbar]]$-algebra topologically generated by $x_{i,m}^+$ (resp. $x_{i,m}^-$; resp. $h_{i,m}$, $c$), and subject to the relations (Q6-Q9) with ``$+$'' (resp. (Q6-Q9) with ``$-$''; resp. (Q1-Q2)).
\end{thm}

The rest of this section is devoted to a proof of Theorem \ref{thm:tri-decomp}.
We first introduce some algebras related to $\U$ that will be used later on.

\begin{de}\label{de:algebras}
Let $\wt\U$ be the $\C[[\hbar]]$-algebra topologically generated by the elements in \eqref{eq:gen-set} with defining relations (Q1-Q5),
$\wh\U$ the quotient algebra of $\wt\U$ modulo the relations (Q6-Q7),
and $\bar\U$ the quotient algebra of $\wh\U$ modulo the relation (Q8).
\end{de}

Denote by $\wt \U^+$ (resp. $\wt \U^-$; resp. $\wt{\mathcal H}$) the closed subalgebra of $\wt\U$ generated by $x_{i,m}^+$ (resp. $x_{i,m}^-$; resp. $h_{i,m}$, $c$).
The following result is standard.

\begin{lem}\label{lem:tri-tech0}
$(\wt\U^-, \wt{\mathcal H},\wt\U^+)$ is a triangular decomposition of $\wt\U$. Moreover, $\wt\U^+$ (resp. $\wt\U^-$)
isomorphic to the $\C[[\hbar]]$-algebra topologically free
generated by $x_{i,m}^+$ (resp. $x_{i,m}^-$) and $\wt{\mathcal H}$ is isomorphic to the $\C[[\hbar]]$-algebra topologically generated by  $h_{i,m}$, $c$ with relations (Q1-Q2).
\end{lem}

The following result was proved in (the proof of) \cite[Lemma 8]{He-representation-coprod-proof}.

\begin{lem}\label{lem:tri-tech1}
For $i,j,k=0,1$, the following  hold in $\wt\U$:
\begin{align}\label{eq:tri-tech1}
  [(z-q^{\pm a_{ij}}w)x_i^\pm(z)x_j^\pm(w)-(q^{\pm a_{ij}}z-w)x_j^\pm(w)x_i^\pm(z),x_k^\mp(w_0)]=0.
\end{align}
\end{lem}

Similarly, we have:

\begin{lem}\label{lem:tri-tech2}
For $i,j,k=0,1$ with $i\ne j$, the following  hold in $\wh\U$:
\begin{align}\label{eq:tri-tech2-temp1}
  &[[x_i^\pm(z_1),(z_2-q^{\mp 2}w)x_i^\pm(z_2)x_j^\pm(w)
    -(q^{\mp 2}z_2-w)x_j^\pm(w)x_i^\pm(z_2)],x_k^\mp(w_0)]=0.
\end{align}
\end{lem}

\begin{proof}Let $i,j$ be as in lemma. We first prove that for $\eta=\pm$
\begin{align}\label{eq:tri-tech2-temp2}[\phi_i^\eta(q^{\mp\frac{\eta}{2}c}z_1),\((z_2-q^{\mp 2}w)x_i^\pm(z_2)x_j^\pm(w)
    -(q^{\mp 2}z_2-w)x_j^\pm(w)x_i^\pm(z_2)\)]=0.
\end{align}
Indeed, it follows from (Q3) and (Q7) that
\begin{align*}
  &[\phi_i^+(q^{\mp\frac{1}{2}c}z_1),\((z_2-q^{\mp 2}w)x_i^\pm(z_2)x_j^\pm(w)
    -(q^{\mp 2}z_2-w)x_j^\pm(w)x_i^\pm(z_2)\)]\\
  =&\qquad\qquad\(\frac{q^{\pm 2}z_1-z_2}{z_1-q^{\pm 2}z_2}
    \frac{q^{\mp 2}z_1-w}{z_2-q^{\mp 2}w}-1\)\\
    &\cdot \((z_2-q^{\mp 2}w)x_i^\pm(z_2)x_j^\pm(w)
    -(q^{\mp 2}z_2-w)x_j^\pm(w)x_i^\pm(z_2)\)
    \phi_i^+(q^{\mp\frac{1}{2}c}z_1)\\
  =&\qquad\qquad\frac{(q^{\pm 2}-q^{\mp 2})z_1(z_2-w)}
    {(z_1-q^{\pm 2}z_2)(z_1-q^{\mp 2}w)}\\
  &\cdot\((z_2-q^{\mp 2}w)x_i^\pm(z_2)x_j^\pm(w)
    -(q^{\mp 2}z_2-w)x_j^\pm(w)x_i^\pm(z_2)\)
    \phi_i^+(q^{\mp\frac{1}{2}c}z_1)\\
  =&0.
\end{align*}
Similarly, for the case $\eta=-$, \eqref{eq:tri-tech2-temp2}   follows from (Q4) and (Q7).
Now, in view of \eqref{eq:tri-tech2-temp2} and (Q5), we have
\begin{align}\label{eq:tri-tech2-temp3}
  &[[x_i^\pm(z_1),x_k^\mp(w_0)],(z_2-q^{\mp 2}w)x_i^\pm(z_2)x_j^\pm(w)
    -(q^{\mp 2}z_2-w)x_j^\pm(w)x_i^\pm(z_2)]=0.
\end{align}
This together with \eqref{eq:tri-tech1} (with $i\ne j$) gives \eqref{eq:tri-tech2-temp1}, which proves the lemma.
\end{proof}

\begin{lem}\label{lem:tri-tech3}
For $i,j,k=0,1$ with $i\ne j$, the following equations hold in $\bar\U$:
\begin{equation}\begin{split}\label{eq:tri-decomp-tech-1}
&\sum_{\sigma\in S_{3}}\sum_{r=0}^3
(-1)^r\binom{3}{r}_{q}
x_i^\pm(z_{\sigma(1)})\cdots x_i^\pm(z_{\sigma(r)})\phi^\eta_j( q^{\mp \eta\frac{1}{2}c}w)\\
&\quad\times
x_i^\pm(z_{\sigma(r+1)})\cdots x_i^\pm(z_{\sigma(3)})=0,
\end{split}
\end{equation}
\begin{equation}\begin{split}\label{eq:tri-decomp-tech-2}
&\sum_{\sigma\in S_3}\sum_{r=0}^3
(-1)^r\binom{3}{r}_{q}
\xi_i(z_{\sigma(1)})\cdots \xi_i(z_{\sigma(r)})x_j^\pm(w)\\
&\quad\times
\xi_i(z_{\sigma(r+1)})\cdots \xi_i(z_{\sigma(3)})=0,
\end{split}\end{equation}
where $\eta=\pm$, $\xi_i(z_p)=x_i^\pm(z_p)$ if $p\ne 1$
and $\xi(z_1)=\phi_i^\eta(q^{\mp \eta \frac{1}{2}c}z_1)$.
\end{lem}

\begin{proof} Equation \eqref{eq:tri-decomp-tech-1} can be  proved as that of
\cite[Eq. (20)]{He-representation-coprod-proof} and we omit the details.
For \eqref{eq:tri-decomp-tech-2}, we will prove the case of  $\eta=+$, the other case $\eta=-$ is similar.
Denote by $R^\pm$ the LHS of \eqref{eq:tri-decomp-tech-2}.
Then it follows from the relations (Q3-Q4) that
\begin{align*}
R^\pm=D^\pm\sum_{\pi\in S_{2}}
\sum_{r=1}^{3}&
    P_{r}(z_{1},z_{\pi(2)},z_{\pi(3)},w,q^{\pm 1})
    x_i^\pm(z_{\pi(2)})\cdots x_i^\pm(z_{\pi(r)})x_j^\pm(w)\\
    &\times x_i^\pm(z_{\pi(r+1)})\cdots x_i^\pm(z_{\pi(3)})
    \phi_i^+(q^{\mp\frac{1}{2}c}z_{1}),
\end{align*}
where   $S_{2}$ acts on the set $\{2,3\}$ and for $1\le r\le 3$,
\begin{align*}
D^\pm&=\frac{1}{z_1-q^{\mp 2}w}
    \prod_{2\le a\le 3}
    \frac{1}{z_1-q^{\pm 2 }z_a},
\end{align*}
and
\begin{align*}
&P_r(z_1,z_2,z_3,w,q)\\
=&\binom{3}{r}_q(-1)^r
    \sum_{p=1}^r \prod_{2\le a\le p}\(z_1-q^2z_a\)
    \prod_{p<a\le 3}\(q^{2}z_1-z_a\)
    \( q^{a_{ij}}z_1-w\)\\
    &\,+\binom{3}{r-1}_q(-1)^{r-1}
    \sum_{p=r}^{3} \prod_{2\le a\le p}\(z_1-q^2z_a\)
    \prod_{p<a\le 3}\(q^2z_1-z_a\)
    \( z_1-q^{a_{ij}}w\).
\end{align*}

It was proved in \cite[Lemma 6]{He-representation-coprod-proof} that
\begin{align}
P_1(z_1,&z_2,z_3,w,q)=(z_2-q^2w)
f^{(1)}_{3}(z_1,z_3,w,q)+(z_{3}-q^{-2}z_{2})
f^{(1)}_2(z_1,w,q),\label{eq:tri-decomp-temp-P1}\\
P_2(z_1,&z_2,z_3,w,q)=(w-q^2z_2)
f^{(2)}_{2}(z_1,z_3,w,q)
+(z_{3}-q^{2}w)
f^{(2)}_{3}(z_1,z_2,w,q)\label{eq:tri-decomp-temp-P2},\\
P_3(z_1,&z_2,z_3,w,q)=(w-q^2z_3)
f^{(3)}_{3}(z_1,z_{2},w,q)+
(z_{3}-q^{-2}z_{2})
f^{(3)}_2(z_1,w,q),\label{eq:tri-decomp-temp-P3}
\end{align}
where
$f_2^{(r)}$ and $f_3^{(r)}$ are some polynomials
of degree at most $1$ in each variable.

In view of \eqref{eq:tri-decomp-temp-P1}, \eqref{eq:tri-decomp-temp-P2}, \eqref{eq:tri-decomp-temp-P3} and
(Q6), we have
\begin{align*}
&\sum_{\pi\in S_{2}}
P_r(z_1,z_{\pi(2)},z_{\pi(3)},w,q^{\pm 1})
 x_i^\pm(z_{\pi(2)})x_i^\pm(z_{\pi(3)}) x_j^\pm(w)\phi_i^+(q^{\mp\frac{1}{2}c}z_{1})=0,\\
&\sum_{\pi\in S_{2}}
P_r(z_1,z_{\pi(2)},z_{\pi(3)},w,q^{\pm 1})
 x_j^\pm(w)x_i^\pm(z_{\pi(2)})x_i^\pm(z_{\pi(3)})
    \phi_i^+(q^{\mp\frac{1}{2}c}z_{1})=0.
\end{align*}
This implies that all the terms in $R^\pm$ which contain the polynomials $f^{(r)}_a$ with $a\ne r,3$ can be erased.
Thus, we  obtain
\begin{align*}
R^\pm=&D^\pm\sum_{\pi\in S_2}(z_{\pi(2)}-q^{\pm 2}w)
    f_3^{(1)}(z_1,z_{\pi(3)},w,q^{\pm 1})
    x_j^\pm(w)x_i^\pm(z_{\pi(2)})x_i^\pm(z_{\pi(3)})
    \phi_i^+(q^{\mp\frac{1}{2}c}z_{1})\\
    -&D^\pm\sum_{\pi\in S_2}(q^{\pm 2}z_{\pi(2)}-w)
    f_2^{(2)}(z_1,z_{\pi(3)},w,q^{\pm 1})
    x_i^\pm(z_{\pi(2)})x_j^\pm(w)x_i^\pm(z_{\pi(3)})
    \phi_i^+(q^{\mp\frac{1}{2}c}z_{1})\\
    +&D^\pm\sum_{\pi\in S_2}(z_{\pi(3)}-q^{\pm 2}w)
    f_3^{(2)}(z_1,z_{\pi(2)},w,q^{\pm 1})
    x_i^\pm(z_{\pi(2)})x_j^\pm(w)x_i^\pm(z_{\pi(3)})
    \phi_i^+(q^{\mp\frac{1}{2}c}z_{1})\\
    -&D^\pm\sum_{\pi\in S_2}(q^{\pm 2}z_{\pi(3)}-w)
    f_3^{(3)}(z_1,z_{\pi(2)},w,q^{\pm 1})
    x_i^\pm(z_{\pi(2)})x_i^\pm(z_{\pi(3)})x_j^\pm(w)
    \phi_i^+(q^{\mp\frac{1}{2}c}z_{1}).
\end{align*}
A straightforward calculation shows that
\begin{align*}
  &f_3^{(1)}(z_1,z_3,w,q)=
  f_2^{(2)}(z_1,z_3,w,q)=Q(z_1-z_3),\\
  &f_3^{(2)}(z_1,z_2,w,q)
  =f_3^{(3)}(z_1,z_2,w,q)=-Q(z_1-z_2),
\end{align*}
where
\begin{align*}
  Q=(q^{-4}-q^{4}+q^{-2}-q^{2})z_1.
\end{align*}
Then we have
\begin{align}
  R^\pm=\pm& Q D^\pm\sum_{\pi\in S_2}(z_{\pi(2)}-q^2w)
    (z_1-z_{\pi(3)})
    x_j^\pm(w)x_i^\pm(z_{\pi(2)})x_i^\pm(z_{\pi(3)})
    \phi_i^+(q^{\mp\frac{1}{2}c}z_{1})\notag\\
    \mp&Q D^\pm\sum_{\pi\in S_2}(q^2z_{\pi(2)}-w)
    (z_1-z_{\pi(3)})
    x_i^\pm(z_{\pi(2)})x_j^\pm(w)x_i^\pm(z_{\pi(3)})
    \phi_i^+(q^{\mp\frac{1}{2}c}z_{1})\notag\\
    \pm&Q D^\pm\sum_{\pi\in S_2}(z_{\pi(3)}-q^2w)
    (z_1-z_{\pi(2)})
    x_i^\pm(z_{\pi(2)})x_j^\pm(w)x_i^\pm(z_{\pi(3)})
    \phi_i^+(q^{\mp\frac{1}{2}c}z_{1})\notag\\
    \mp&Q D^\pm\sum_{\pi\in S_2}(q^2z_{\pi(3)}-w)
    (z_1-z_{\pi(2)})
    x_i^\pm(z_{\pi(2)})x_i^\pm(z_{\pi(3)})x_j^\pm(w)
    \phi_i^+(q^{\mp\frac{1}{2}c}z_{1})\notag\\
    =\mp& Q D^\pm q^{\pm 2}\sum_{\pi\in S_2}(z_1-z_{\pi(2)})\big[x_i^\pm(z_{\pi(2)}),
   (z_{\pi(3)}-q^{\mp 2}w)x_i^\pm(z_{\pi(3)})x_j^\pm(w)\label{newrela1}\\
    &\quad\quad
    -(q^{\mp 2}z_{\pi(3)}-w)x_j^\pm(w)x_i^\pm(z_{\pi(3)})\big]
    \phi_i^+(q^{\mp\frac{1}{2}c}z_{1})\notag\\
    =0&,\notag
\end{align}
where the last equation follows from (Q8).
\end{proof}

As in the proof of \cite[Lemma 10]{He-representation-coprod-proof}, it is obvious that Lemma \ref{lem:tri-tech3} implies the following result.

\begin{lem}\label{lem:tri-tech4}
For $i,j,k=0,1$ with $i\ne j$, the following equations hold in $\bar\U$:
\begin{align*}
[\sum_{\sigma\in S_3}\sum_{r=0}^3(-1)^r\binom{3}{r}_q
   x_i^\pm(z_{\sigma(1)})\cdots x_i^\pm(z_{\sigma(r)})x_j^\pm(w)
   x_i^\pm(z_{\sigma(r+1)})\cdots x_i^\pm(z_{\sigma(3)}),x_k^\mp(w_0)]=0.
   \end{align*}
\end{lem}

\textbf{Proof of Theorem \ref{thm:tri-decomp}:} We first recall a general result of triangular decompositions (cf. \cite[Lemma 3.5]{He-representation-coprod-proof}).
Let $A$ be a completed and separated   $\C[[\hbar]]$-algebra and
$(A^-,H,A^+)$ a triangular decomposition of $A$. Let $B^+$ and $B^-$ be respectively a closed two-sided ideal of
$A^+$ and $A^-$, and let $B$ be the closed ideal of $A$ generated by $B^++B^-$.
Set $C=A/B$ and denote by $C^\pm$ the image of
$B^\pm$ in $C$.
Assume that $A B^+\subset B^+ A$ and $B^- A\subset A B^-$. Then  $(C^+,H,C^-)$ is a triangular decomposition
of $C$ and $C^\pm$ are isomorphic to $A^\pm/B^\pm$.
In view of this criterion, Theorem \ref{thm:tri-decomp}
follows from Lemmas \ref{lem:tri-tech0}, \ref{lem:tri-tech1},
\ref{lem:tri-tech2} and \ref{lem:tri-tech4}.

\begin{rem}{\em It was proved in \cite{He-representation-coprod-proof} that the algebra $\U_\hbar(\dot\g_{tor})$ has a  triangular decomposition
as in Theorem \ref{thm:tri-decomp}.
That is, $\U$ and $\U_\hbar(\dot\g_{tor})$ are two different choices of the quotient algebras of $\U_\hbar(\hat\g)$ with
a triangular decomposition.
}
\end{rem}

\section{Hopf algebra structure}

In this section, we discuss a 
Hopf algebra structure on $\U$.
For a $\C[[\hbar]]$-module $M$ and $n\in \mathbb N$, we denote that 
\begin{align}
  M^{\wh\ot^n}=
  \belowit{\underbrace{M\wh\ot M \wh\ot\cdots\wh\ot M}}
  {n\te{-copies}}.
\end{align} Let $u,v$ be formal variables.  Motivated by the deformed Drinfeld coproduct given in \cite{He-drinfeld-coproduct}, we have:
\begin{thm}\label{thm:bialg-struct}
There exists a unique $\C[[\hbar]]$-algebra homomorphism
$\Delta_u:\U\rightarrow\big(\U^{\wh\ot^2}\big) ((u))$ defined as follows $(i=0,1)$
\begin{eqnarray*}
&&\te{(Co1) }\qquad\qquad\Delta_u(c)=c\ot 1+1\ot c,\qquad\qquad\\
&&\te{(Co2) }\qquad\qquad\Delta_u\left(\phi_i^\pm(z)\right)
    =\phi_i^\pm(zq^{\pm\frac{c_2}{2}})\otimes \phi_i^\pm(zu\inverse q^{\mp\frac{c_1}{2}}),\qquad\qquad\\
&&\te{(Co3) }\qquad\qquad\Delta_u\left(x_i^+(z)\right)
    =x_i^+(z)\otimes 1+\phi_i^-(zq^{\frac{c_1}{2}})\otimes x_i^+(zu\inverse q^{c_1}),\qquad\qquad\\
&&\te{(Co4) }\qquad\qquad\Delta_u\left(x_i^-(z)\right)
    =1\otimes x_i^-(zu\inverse)+x_i^-(zq^{c_2})\otimes \phi_i^+(zu\inverse q^{\frac{c_2}{2}}),\qquad\qquad
\end{eqnarray*}
where $c_1=c\ot 1$ and $c_2=1\otimes c$.
Moreover, as $\C[[\hbar]]$-algebra homomorphisms
$\U\rightarrow\big(\U^{\wh\ot^3}\big)((u,v))$,
we have
\begin{align*}
  &\(\mathrm{Id}\otimes \Delta_v\)\circ \Delta_u
  =\(\Delta_u\otimes \mathrm{Id}\)\circ \Delta_{uv}.
\end{align*}
and, as $\C[[\hbar]]$-algebra homomorphisms
$\U\rightarrow\big(\U^{\wh\ot^2}\big)((u))$,
we have
\begin{align*}
  &\(\mathrm{Id}\otimes \epsilon\)\circ\Delta_u=\mathrm{Id},\quad
  \(\epsilon\otimes \mathrm{Id}\)\circ\Delta_u=\mathrm{Id}_u,
\end{align*}
where $\mathrm{Id}_u:\U\rightarrow \U\otimes \C((u))$ is the $\C[[\hbar]]$-algebra
homomorphism determined by
\begin{eqnarray*}
&\te{(Id1) }\qquad\qquad&\mathrm{Id}_u(c)=c,\quad \mathrm{Id}_u(\phi_i^\pm(z))=\phi_i^\pm(zu\inverse),\qquad\qquad\\
&\te{(Id2) }\qquad\qquad&\mathrm{Id}_u(x_i^\pm(z))=x_i^\pm(zu\inverse),\qquad\qquad
\end{eqnarray*}
and $\epsilon:\U\rightarrow \C[[\hbar]]$ is the $\C[[\hbar]]$-algebra homomorphism  determined by
\begin{eqnarray*}
&\te{(CoU) }\qquad\qquad \epsilon\(\phi_i^\pm(z)\)=1,\quad
\epsilon\(x_i^\pm(z)\)=0=\epsilon(c).\qquad\qquad
\end{eqnarray*}
\end{thm}

Before proving Theorem \ref{thm:bialg-struct}, we remark that the above 
gives a (topological) Hopf algebra structure on $\U$ (by $\Delta_1$, the ``limit" of $\Delta_u$ at $u=1$).
However, one notices that
$\Delta_1$ is not a well-defined $\C[[\hbar]]$-algebra homomorphism from $\U$ to $\U^{\wh\ot^2}$, so
we need to introduce certain topological completions of $\U$ and $\U^{\wh\ot^2}$ as in \cite{Gr-qshuff-qaff}.
Explicitly,
let $\mathcal F$ be the free $\C[[\hbar]]$-algebra topologically generated by the set \eqref{eq:gen-set}.
Now give $h_{i,\pm n}$ degree $n$ for $n>0$, and
all other elements degree $0$. We extend the degree to all the elements of the algebra by summation on the monomials.
For $k\ge0$, let $\mathcal F_k$ be the $\hbar$-adically closed ideal of $\mathcal F$ generated by elements of degree greater than $k$.
Then we obtain an inverse system of $\C[[\hbar]]$-algebras $(\mathcal F/\mathcal F_k, p_k)$,
where $p_k$ is the natural projection $\mathcal F/\mathcal F_k\to \mathcal F/\mathcal F_{k-1}$.
Denote by $\mathcal F_c$ the $\hbar$-adic completion of the inverse limit
$\varprojlim\mathcal F/\mathcal F_k$.
Note that $\mathcal F_c$ is a complete and separated algebra over $\C[[\hbar]]$ with inverse limit topology.
Let $K$ be the closed ideal of $\mathcal F_c$ generated by the relations (Q1-Q10).
Set
\begin{align*}
  \U_c=\mathcal F_c/K,
\end{align*}
 a completion of $\U$.
Note that there is a canonical injection from $\U$ to $\U_c$.

Now, we consider the space $\U^{\wh\ot^2}$. In this case, we view $\mathcal F$ as a $\Z$-graded algebra by giving $x_{i,\pm n}^+$, $x_{i,\pm n}^-$, $h_{i,\pm n}$ degree $n$ for $n\ge 0$, and give other generators degree $0$.
Denote by $\mathcal F_k'$ the closed two sided ideal of $\mathcal F$ of elements of degree at least $k$.
One notices that $\mathcal F_k'$ is a strict subset of $\mathcal F_k$.
Let $\mathcal F\bar\ot\mathcal F$ be the topological completion of the inverse limit
\begin{align*}
  \mathcal F\wh\ot \mathcal F/\mathcal F_k'\wh\ot\mathcal F_k'.
\end{align*}
Then  $\mathcal F\bar\ot\mathcal F$ is also a complete and separated algebra over $\C[[\hbar]]$.
Define $\U_c\wt\ot\U_c$ to be the quotient algebra of $\mathcal F\bar\ot\mathcal F$ modulo the closure of $K\wh\ot\mathcal F+\mathcal F\wh\ot K$.
It is easy to see that there is a canonical injection from $\U\wh\ot \U$ to $\U_c\wt\ot\U_c$.
Using these completions, we deduce from Theorem \ref{thm:bialg-struct} that
$(\U_c,\Delta_1,\epsilon)$ carries a $\C[[\hbar]]$-bialgebra structure.
Furthermore, by the same argument as in the proof of
\cite[Theorem 2.1]{DI-generalization-qaff}, we have the following result.

\begin{thm}\label{thm:hopf}
$\U_c$ is a Hopf algebra with coproduct $\Delta_1$,
counit $\epsilon$
and the antipode $S$ defined by $(i=0,1)$
\begin{align*}
  &S(c)=-c,\qquad
  S\(x_i^+(z)\)=-\phi_i^-(zq^{-\frac{c}{2}})\inverse x_i^+(zq^{-c}),\\
  &S\(x_i^-(z)\)=-x_i^-(zq^{-c})\phi_i^+(zq^{-\frac{c}{2}}),\quad  S\(\phi_i^\pm(z)\)=\phi_i^\pm(z)\inverse.
\end{align*}
\end{thm}

The rest of this section is devoted to proving Theorem \ref{thm:bialg-struct}.
Recall the algebras $\wh\U$ and $\bar\U$ introduced in Definition \ref{de:algebras}.
Firstly, we have the following straightforward result.

\begin{lem}\label{lem:coprod-Q1-Q8}
(Co1-Co4) defines a unique $\C[[\hbar]]$-algebra homomorphism $\wh\Delta_u:\wh\U\to\big(\wh\U^{\wh\ot^2}\big)((u))$.
\end{lem}

Furthermore,  we have

\begin{lem}\label{lem:coprod-Q9}
$\wh\Delta_u$ induces a $\C[[\hbar]]$-algebra homomorphism from $\bar\Delta_u:\bar\U\to\(\bar\U^{\wh\ot^2}\)((u))$.
\end{lem}

\begin{proof}
Fix any $i\ne j\in\{0,1\}$ and denote by
 $I_{ij}^\pm$ the LHS of the relation (Q8).
We need to  show that $\wh\Delta_u(I_{ij}^\eta)=0$ with $\eta=\pm$. We will show
the case $\eta=+$, as the case $\eta=-$ is similar and thus omitted.
Set
\begin{align*}
  &x_{ij}^+(z,w)=(z-q^{-2}w)x_i^+(z)x_j^+(w)-(q^{-2}z-w)x_j^+(w)x_i^+(z),\\
  &A_{ij}^+(z,w)=(z-q^{-2}w)\wh\Delta_u(x_i^+(z))\wh\Delta_u(x_j^+(w))
  -(q^{-2}z-w)\wh\Delta_u(x_j^+(w))\wh\Delta_u(x_i^+(z)).
\end{align*}
A straightforward calculation shows that
\begin{align*}
  &A_{ij}^+(z,w)
  =x_{ij}^+(z,w)\ot 1
  +q^{-c_1}u\phi_i^-(zq^{\frac{c_1}{2}})\phi_j^-(wq^{\frac{c_1}{2}})
  \ot x_{ij}^+(zu\inv q^{c_1},wu\inv q^{c_1}).
\end{align*}
Using this, we obtain
\begin{align}
  &[\wh\Delta_u(x_i^+(z_1)),A_{ij}^+(z_2,w)]
  =[x_i^+(z_1),x_{ij}^+(z_2,w)]\ot 1\label{eq:coprod-Q9-temp1}\\
  +&[\phi_i^-(z_1q^{\frac{c_1}{2}}),x_{ij}^+(z_2,w)]\ot x_i^+(z_1u\inv q^{c_1})
  +q^{-c_1}u\phi_i^-(z_1q^{\frac{c_1}{2}})\phi_i^-(z_2q^{\frac{c_1}{2}})
  \phi_j^-(wq^{\frac{c_1}{2}})\nonumber\\
  &\ot[x_i^+(z_1u\inv q^{c_1}),x_{ij}^+(z_2u\inv q^{c_1},wu\inv q^{c_1})]\nonumber\\
  &+q^{-c_1}u[x_i^+(z_1),\phi_i^-(z_2q^{\frac{c_1}{2}})
    \phi_j^-(wq^{\frac{c_1}{2}})]
  \ot x_{ij}^+(z_2u\inv q^{c_1},wu\inv q^{c_1})\nonumber
\end{align}
By applying (Q4) one gets that
\begin{align}
  &[\phi_i^-(z_1q^{\frac{c}{2}}),x_{ij}^+(z_2,w)]\nonumber\\
  =&\phi_i^-(z_1q^{\frac{c}{2}})x_{ij}^+(z_2,w)
  \(1- g_{ji}(z_1/w)g_{ii}(z_1/z_2) \)\nonumber\\
  =&\frac{(q^2-q^{-2})z_1(z_2-w)}{(w-q^{-2}z_1)(z_2-q^2z_1)}
  \phi_i^-(z_1q^{\frac{c}{2}})x_{ij}^+(z_2,w),\label{eq:coprod-Q9-temp2}\\
  &[x_i^+(z_1),\phi_i^-(z_2q^{\frac{c}{2}})
    \phi_j^-(wq^{\frac{c}{2}})]\nonumber\\
  =&\phi_i^-(z_2q^{\frac{c}{2}})
    \phi_j^-(wq^{\frac{c}{2}})x_i^+(z_1)
    \(g_{ji}(z_1/w)g_{ii}(z_1/z_2)-1\)\nonumber\\
  =&-\frac{(q^2-q^{-2})z_1(z_2-w)}{(w-q^{-2}z_1)(z_2-q^2z_1)}
  \phi_i^-(z_2q^{\frac{c}{2}})
    \phi_j^-(wq^{\frac{c}{2}})x_i^+(z_1).\label{eq:coprod-Q9-temp3}
\end{align}
Recall from (Q7) that
\begin{align}\label{eq:coprod-Q9-temp4}
  (z-w)x_{ij}^+(z,w)=0=(z-w)x_{ij}^+(zu\inv q^{c_1},wu\inv q^{c_1}).
\end{align}
Combining  \eqref{eq:coprod-Q9-temp1}, \eqref{eq:coprod-Q9-temp2}, \eqref{eq:coprod-Q9-temp3} and \eqref{eq:coprod-Q9-temp4}, we deduce from (Q8) that
\begin{align*}
  [\wh\Delta_u(x_i^+(z_1)),A_{ij}^+(z_2,w)]=0.
\end{align*}
This implies that $\wh\Delta_u(I_{ij}^+)=0$, as required.
\end{proof}

To continue the discussion, we need to introduce some notations.
For  $0\le s\le 3$, we denote by $S_{3,s}$ the set of $(s,3-s)$-shuffles in $S_3$, that is
\begin{align*}
  S_{3,s}=\set{\sigma\in S_3}{\sigma(a)<\sigma(b),\,\te{for }a<b\le s\,\,\te{or}\,\,s<a<b}.
\end{align*}
As a convention, we let $\sigma(a)=a$ for any $a<1$ or $a>3$.
For $0\le s\le 3$ and $\sigma\in S_{3,s}$, we define two partitions
\begin{align*}
P_{0,\sigma}^1\cup \cdots \cup P_{3-s,\sigma}^1\quad\te{and}\quad
P_{0,\sigma}^1\cup \cdots \cup P_{s+1,\sigma}^1
\end{align*}
 of the set $\{0,1,2,3\}$,
where
\begin{align*}
  &P_{0,\sigma}^1=\set{p\in\Z}{0\le p<\sigma(s+1)},\\
  &P_{k,\sigma}^1=\set{p\in\Z}{\sigma(s+k)\le p<\sigma(s+k+1)}\quad\te{for } 0<k<3-s,\\
  &P_{3-s,\sigma}^1=\set{p\in\Z}{\sigma(3)\le p\le 3}\quad\te{if }s<3,\\
  &P_{0,\sigma}^2=\set{p\in\Z}{0\le p<\sigma(1)}\quad\te{if }s>0,\\
  &P_{k,\sigma}^2=\set{p\in\Z}{\sigma(k)\le p<\sigma(k+1)}\quad\te{for } 0<k<s,\\
  &P_{s,\sigma}^2=\set{p\in\Z}{\sigma(s)\le p\le 3}.
\end{align*}
Furthermore, for $0\le s\le 3$ and $0\le k\le 3-s$, set
\begin{align*} &T_{s,k}^1(z_1,z_2,z_3,w)=
  \sum_{\sigma\in S_{3,s}}\sum_{r\in P_{k,\sigma}^1}\binom{3}{r}_q(-1)^{r}
  \prod_{\substack{a\le s<b, \sigma(a)<\sigma(b)}}(q^2z_{a}-z_{b})\\
  &\cdot \prod_{\substack{a\le s<b, \sigma(a)>\sigma(b)}}(z_{a}-q^2z_{b})
  \prod_{\substack{a\le s, \sigma(a)\le r}}(q^{-2}z_{a}-w)
  \prod_{\substack{a\le s, \sigma(a)> r}}(z_{a}-q^{-2}w),
\end{align*}
and for $0\le s\le 3$ and $0\le k\le s$, set
\begin{align*}
  &T_{s,k}^2(z_1,z_2,z_3,w)=
  \sum_{\sigma\in S_{3,s}}\sum_{r\in P_{k,\sigma}^2}\binom{3}{r}_q(-1)^{r}
  \prod_{\substack{a\le s<b, \sigma(a)<\sigma(b)}}(q^2z_{a}-z_{b})\\
  &\cdot \prod_{\substack{a\le s<b, \sigma(a)>\sigma(b)}}(z_{a}-q^2z_{b})
\prod_{\substack{a> s, \sigma(a)>r}}(q^{-2}w-z_{a})
  \prod_{\substack{a> s, \sigma(a)\le r}}(w-q^{-2}z_{a}).
\end{align*}
We have:
\begin{lem}\label{lem:coprod-Q10-2}
\textrm{(1)} For  $0\le k\le 3$, one has that
  $$T_{0,k}^1(z_1,z_2,z_3,w)=\binom{3}{k}_q(-1)^k.$$

\textrm{(2)}
There exist  polynomials $f_{1,0}(z_1,z_2),f_{1,1}(z_1,z_2),f_{1,2}(z_1,z_2)\in\C[[\hbar]][z_1,z_2]$  such that
\begin{align*}
  &T_{1,0}^1(z_1,z_2,z_3,w)
  =f_{1,0}(z_1,w)(z_2-q^2 z_3)-f_{1,1}(z_1,z_3)(q^{-2}z_2-w),\\
  &T_{1,1}^1(z_1,z_2,z_3,w)
  =f_{1,1}(z_1,z_3)(z_2-q^{-2}w)+f_{1,1}(z_1,z_2)(q^{-2}z_3-w),\\
  &T_{1,2}^1(z_1,z_2,z_3,w)
  =-f_{1,1}(z_1,z_2)(z_3-q^{-2}w)+f_{1,2}(z_1,w)(z_2-q^2z_3).
\end{align*}
.

\textrm{(3)} There exist polynomials  $$f_{2,0}(z_1,z_2,z_3,w),f_{2,1}(z_1,z_2),f_{2,2}(z_1,z_2,z_3,w)\in\C[[\hbar]][z_1,z_2,z_3,w]$$
such that  $$f_{2,k}(z_1,z_2,z_3,w)=f_{2,k}(z_2,z_1,z_3,w),\quad k=0,2,$$
and
\begin{align*}
  &T_{2,0}^1(z_1,z_2,z_3,w)=f_{2,0}(z_1,z_2,z_3,w)(z_1-q^2 z_2)
  +f_{2,1}(z_1,z_2)(z_3-z_2)(q^{-2}z_3-w),\\
  &T_{2,1}^1(z_1,z_2,z_3,w)=f_{2,2}(z_1,z_2,z_3,w)(z_1-q^2z_2)
  -f_{2,1}(z_1,z_2)(z_3-z_2)(z_3-q^{-2}w).
\end{align*}

\textrm{(4)} There exist polynomials $f_{3,0}(z_1,z_2),f_{3,1}(z_1,z_2)\in\C[[\hbar]][z_1,z_2]$ such that
\begin{align*}
  &T_{3,0}^1(z_1,z_2,z_3,w)=f_{3,0}(z_3,w)(z_1-q^2z_2)+f_{3,1}(z_1,w)(z_2-q^2w).
\end{align*}

\textrm{(5)} $T_{s,k}^2(z_1,z_2,z_3,w)=(-1)^{s}T_{3-s,k}^1(z_3,z_1,z_2,w)$ if $s<3$.

\textrm{(6)}
$T_{3,0}^1(z_1,z_2,z_3,w)=-T_{0,0}^2(z_1,z_2,z_3,w)$.
\end{lem}

\begin{proof}
The lemma follows from the following straightforward facts:
\begin{align*}
  &T_{0,k}^1(z_1,z_2,z_3,w)=T_{3,k}^2(z_1,z_2,z_3,w)=\binom{3}{k}_q(-1)^k,\\
  &T_{1,0}^1(z_1,z_2,z_3,w)=-T_{2,0}^2(z_2,z_3,z_1,w)\\
  &\quad=(q^{4}-1)z_1\big((q^{-4}z_1-w)(z_2-q^2z_3)
    -[3]_{q^2}(z_1-z_3)(q^{-2}z_2-w)\big),\\
  &T_{1,1}^1(z_1,z_2,z_3,w)=-T_{2,1}^2(z_2,z_3,z_1,w)\\
  &\quad=(q^4-1)z_1[3]_{q^2}\big((z_1-z_3)(z_2-q^{-2}w)
  +(z_1-z_2)(q^{-2}z_3-w)\big),\\
  &T_{1,2}^1(z_1,z_2,z_3,w)=-T_{2,2}^2(z_2,z_3,z_1,w)\\
  &\quad=(q^4-1)z_1\big(
    -[3]_{q^2}(z_1-z_2)(z_3-q^{-2}w)
    +(z_1-q^{-4}w)(z_2-q^2z_3)
  \big),\\
  &T_{2,0}^1(z_1,z_2,z_3,w)=T_{1,0}^2(z_2,z_3,z_1,w)\\
  &\quad=z_1z_2(1-q^{-2})(q^4-1)[3]_{q^2}(z_3-w)(q^{-2}z_3-w)
  +(1-q^{-4})(z_1-q^2z_2)\\
  &\quad\times\big(q^{2}[3]_{q^2}(z_3-w)(q^4z_1z_2-z_3w)
  -z_3(1+q^{-2})(q^2z_1-z_3)(q^2z_2-z_3)\big),\\
  &T_{2,1}^1(z_1,z_2,z_3,w)=T_{1,1}^2(z_2,z_3,z_1,w)\\
  &\quad=-z_1z_2(1-q^{-2})(q^4-1)[3]_{q^2}(z_3-w)(z_3-q^{-2}w)
  -(1-q^{-4})(z_1-q^2z_2)\\
  &\quad\times\big(q^{2}[3]_{q^2}(z_3-w)(z_1z_2-q^4z_3w)
  -z_3(1+q^{-2})(z_1-q^2z_3)(z_2-q^2z_3)\big),\\
  &T_{3,0}^1(z_1,z_2,z_3,w)=-T_{0,0}^2(z_1,z_2,z_3,w)\\
  &\quad=w(1-q^{-4})\big(-(z_3-q^{-4}w)
 (z_1-q^2 z_2)+(q^{-4}z_1-w)(z_2-q^2z_3)\big).
\end{align*}
\end{proof}

 For $i\ne j\in \{0,1\}$, denote by $J_{ij}^\pm$  the
LHS of (Q9).
\begin{lem}\label{lem:coprod-Q10-1}
 For $i\ne j\in \{0,1\}$, we have
\begin{align*}
  \bar\Delta_u(J_{ij}^+)
  =&\sum_{\tau\in S_3}\sum_{s=0}^3\sum_{k=0}^{3-s}
  \frac{T_{s,k}^1(z_{\tau(1)},z_{\tau(2)},z_{\tau(3)},w)}
  {\prod_{a\le s<b}(z_{\tau(a)}-q^2z_{\tau(b)})
  \prod_{a\le s}(z_{\tau(a)}-q^{-2}w)}\\
  \cdot&
  \wt\phi_i^-(z_{\tau(s+1)})\cdots \wt\phi_i^-(z_{\tau(3)})\wt\phi_j^-(w)
  x_i^+(z_{\tau(1)})\cdots x_i^+(z_{\tau(s)})\\
  &\ot
  \wt x_i^+(z_{\tau(s+1)})\cdots \wt x_i^+(z_{\tau(k)})\wt x_j^+(w)
  \wt x_i^+(z_{\tau(k+1)})\cdots \wt x_i^+(z_{\tau(3)})\\
  +\sum_{\tau\in S_3}&\sum_{s=0}^3\sum_{k=0}^{s}
  \frac{T_{s,k}^2(z_{\tau(1)},z_{\tau(2)},z_{\tau(3)},w)}
    {\prod_{a\le s<b}(z_{\tau(a)}-q^2z_{\tau(b)})
    \prod_{a> s}(w-q^{-2}z_{\tau(a)})}\\
  \cdot&
  \wt\phi_i^-(z_{\tau(s+1)})\cdots \wt\phi_i^-(z_{\tau(3)})
  x_i^+(z_{\tau(1)})\cdots x_i^+(z_{\tau(k)})x_j^+(w)\\
  \cdot&
  x_i^+(z_{\tau(k+1)})
  \cdots x_i^+(z_{\tau(s)})
  \ot \wt x_i^+(z_{\tau(s+1)})\cdots \wt x_i^+(z_{\tau(3)}),
\end{align*}
where $\wt\phi_k^-(z)=\phi_k^-(zq^{\frac{c_1}{2}})$ and $\wt x_k^+(z)=x_k^+(zu\inv q^{c_1})$ for $k=i,j$.
\end{lem}

\begin{proof}
For $\sigma\in S_{3,s}$, we define
\begin{align*}
  \xi_\sigma(z_a)=\begin{cases}
           x_i^+(z_a),&\te{if }\sigma\inv(a)\le s,\\
           \wt\phi_i^-(z_a),&\te{if }\sigma\inv(a)>s,
         \end{cases},
  \quad
  \wt\xi_\sigma(z_a)=\begin{cases}
                \wt x_i^+(z_a),&\te{if }\sigma\inv(a)>s,\\
                1,&\te{if }\sigma\inv(a)\le s.
              \end{cases}
\end{align*}
For $0\le r,s\le 3$ and $\sigma\in S_{3,s}$, set
\begin{align*}
  \xi_{r,\sigma}(z_1,z_2,z_3,w)=&\xi_\sigma(z_1)\cdots \xi_\sigma(z_r)\phi_j^-(wq^{\frac{c}{2}})\xi_\sigma(z_{r+1})
  \cdots \xi(z_3),\\
  \eta_{r,\sigma}(z_1,z_2,z_3,w)=&\xi_\sigma(z_1)\cdots \xi_\sigma(z_r)x_j^+(w)\xi_\sigma(z_{r+1})
  \cdots \xi_\sigma(z_3),\\
  \wt \xi_\sigma(z_1,z_2,z_3)=&\wt\xi_\sigma(z_1)\wt\xi_\sigma(z_2)\wt\xi_\sigma(z_3),\\
  \wt \eta_{r,\sigma}(z_1,z_2,z_3,w)=&\wt\xi_\sigma(z_1)\cdots \wt\xi_\sigma(z_r)x_j^+(wu\inv q^{c_1})\wt\xi_\sigma(z_{r+1})
  \cdots \wt\xi_\sigma(z_3).
\end{align*}
And for $i_1,i_2,\dots,i_n\in \{0,1\}$, $\zeta_{i_r}(z_r)=x_{i_r}^+(z_r)$ or $\phi_{i_r}^-(z_rq^{\frac{c}{2}})$, we define an ordered product
\begin{align*}
  \:\zeta_{i_1}(z_1)\cdots \zeta_{i_n}(z_n)\;
\end{align*}
by moving $\phi_{i_r}^-(z_rq^{\frac{c}{2}})$ to the left.
Then it follows from (Q4) that
\begin{align*}
  &\xi_{r,\sigma}(z_1,z_2,z_3,w)=\prod_{\substack{ a\le s<b}}
  g_{ii}(z_{\sigma(b)}/z_{\sigma(a)})
  \prod_{\substack{a\le s\\ \sigma(a)\le r}}g_{ij}(w/z_{\sigma(a)})
  \:\xi_{r,\sigma}(z_1,z_2,z_3,w)\;,\\
  &\eta_{r,\sigma}(z_1,z_2,z_3,w)=\prod_{\substack{ a\le s<b}}
  g_{ii}(z_{\sigma(b)}/z_{\sigma(a)})
  \prod_{\substack{a\le s\\ \sigma(a)> r}}g_{ji}(z_{\sigma(a)}/w)
  \:\eta_{r,\sigma}(z_1,z_2,z_3,w)\;.
\end{align*}
Moreover, it is straightforward to see that
\begin{align*}
  &\Delta_u(x_i^+(z_1))\cdots\Delta_u(x_i^+(z_r))\Delta_u(x_j^+(w))
  \Delta_u(x_i^+(z_{r+1}))\cdots\Delta_u(x_i^+(z_3))\\
  =&\sum_{s=0}^3\sum_{\sigma\in S_{3,s}}
  \prod_{\substack{ a\le s<b}}
  g_{ii}(z_{\sigma(b)}/z_{\sigma(a)})
  \prod_{\substack{a\le s\\ \sigma(a)\le r}}g_{ij}(w/z_{\sigma(a)})
  \:\xi_{r,\sigma}(z_1,z_2,z_3,w)\;\ot\wt\eta_{r,\sigma}(z_1,z_2,z_3,w)\\
  &+\sum_{s=0}^3\sum_{\sigma\in S_{3,s}}
  \prod_{\substack{ a\le s<b}}
  g_{ii}(z_{\sigma(b)}/z_{\sigma(a)})
  \prod_{\substack{a\le s\\ \sigma(a)> r}}g_{ji}(z_{\sigma(a)}/w)
  \:\eta_{r,\sigma}(z_1,z_2,z_3,w)\;\ot\wt\xi_\sigma(z_1,z_2,z_3).
\end{align*}

Now the lemma follows from a direct calculation and the following facts:
\begin{align*}
  &\:\xi_{r_1,\sigma}(z_{\tau(1)},z_{\tau(2)},z_{\tau(3)},w)\;\ot \wt\eta_{r_1,\sigma}(z_{\tau(1)},z_{\tau(2)},z_{\tau(3)},w)\\
  =&\:\xi_{0,1}(z_{\tau\sigma(1)},z_{\tau\sigma(2)},z_{\tau\sigma(3)},w)\;\ot \wt\eta_{p_{k,\sigma}^1,1}(z_{\tau\sigma(1)},z_{\tau\sigma(2)},z_{\tau\sigma(3)},w),\\
  &\:\eta_{r_2,\sigma}(z_{\tau(1)},z_{\tau(2)},z_{\tau(3)},w)\;\ot \wt\xi_{\sigma}(z_{\tau(1)},z_{\tau(2)},z_{\tau(3)},w)\\
  =&\:\eta_{p_{k,\sigma}^2,1}(z_{\tau\sigma(1)},z_{\tau\sigma(2)},z_{\tau\sigma(3)},w)\;\ot \wt\xi_{1}(z_{\tau\sigma(1)},z_{\tau\sigma(2)},z_{\tau\sigma(3)},w),
\end{align*}
where $\tau\in S_3$, $\sigma\in S_{3,s}$ ($0\le s\le 3$), $r_a\in P_{k,\sigma}^a$ ($a=1,2$) and
$p_{k,\sigma}^a$ the minimal element in $P_{k,\sigma}^a$ ($a=1,2$).

\end{proof}

\textbf{Proof of Theorem \ref{thm:bialg-struct}}:
In view of Lemmas \ref{lem:coprod-Q1-Q8} and \ref{lem:coprod-Q9}, it suffices to show that
\begin{align*}
  \bar\Delta_u(J_{ij}^\pm)=0,\quad \te{for }i\ne j\in\{0,1\}.
\end{align*}
Let  $i\ne j\in\{0,1\}$ be fixed.
From Lemma \ref{lem:coprod-Q10-1} and (1), (5) of Lemma \ref{lem:coprod-Q10-2}, it follows  that
\begin{align*}
  \bar\Delta_u(J_{ij}^+)
  =&\sum_{\tau\in S_3}\sum_{s=1}^3\sum_{k=0}^{3-s}
  \frac{T_{s,k}^1(z_{\tau(1)},z_{\tau(2)},z_{\tau(3)},w)}
  {\prod_{a\le s<b}(z_{\tau(a)}-q^2z_{\tau(b)})
  \prod_{a\le s}(z_{\tau(a)}-q^{-2}w)}\\
  \cdot&
  \wt\phi_i^-(z_{\tau(s+1)})\cdots \wt\phi_i^-(z_{\tau(3)})\wt\phi_j^-(w)
  x_i^+(z_{\tau(1)})\cdots x_i^+(z_{\tau(s)})\\
  &\ot
  \wt x_i^+(z_{\tau(s+1)})\cdots \wt x_i^+(z_{\tau(k)})\wt x_j^+(w)
  \wt x_i^+(z_{\tau(k+1)})\cdots \wt x_i^+(z_{\tau(3)})\\
  +\sum_{\tau\in S_3}&\sum_{s=0}^2\sum_{k=0}^{s}
  \frac{T_{s,k}^2(z_{\tau(1)},z_{\tau(2)},z_{\tau(3)},w)}
    {\prod_{a\le s<b}(z_{\tau(a)}-q^2z_{\tau(b)})
    \prod_{a> s}(w-q^{-2}z_{\tau(a)})}\\
  \cdot&
  \wt\phi_i^-(z_{\tau(s+1)})\cdots \wt\phi_i^-(z_{\tau(3)})
  x_i^+(z_{\tau(1)})\cdots x_i^+(z_{\tau(k)})x_j^+(w)\\
  \cdot&
  x_i^+(z_{\tau(k+1)})
  \cdots x_i^+(z_{\tau(s)})
  \ot \wt x_i^+(z_{\tau(s+1)})\cdots \wt x_i^+(z_{\tau(3)}).
\end{align*}
Combining this with Lemma \ref{lem:coprod-Q10-2} and (Q6-Q7), a similar argument of Lemma \ref{lem:tri-tech3} shows that
\begin{equation}\begin{split}\label{newrela2}
  \bar\Delta_u(J_{ij}^+)
  &=\sum_{\tau\in S_3}
  \frac{f_{1,1}(z_{\tau(1)},z_{\tau(3)})}
  {\prod_{a\le s<b}(z_{\tau(a)}-q^2z_{\tau(b)})
  \prod_{a\le s}(z_{\tau(a)}-q^{-2}w)}\wt\phi_i^-(z_{\tau(2)})\wt\phi_i^-(z_{\tau(3)})\wt\phi_j^-(w)\\
  \cdot&
  x_i^+(z_{\tau(1)})
  \ot\big[(z_{\tau(2)}-q^{-2}w)\wt x_i^+(z_{\tau(2)})\wt x_j^+(w)
    -(q^{-2}z_{\tau(2)}-w)\wt x_j^+(w)
  \wt x_i^+(z_{\tau(2)}),\wt x_i^+(z_{\tau(3)})\big]\\
  +\sum_{\tau\in S_3}&
  \frac{f_{1,1}(z_{\tau(3)},z_{\tau(2)})}
    {\prod_{a\le s<b}(z_{\tau(a)}-q^2z_{\tau(b)})
    \prod_{a> s}(w-q^{-2}z_{\tau(a)})}\wt\phi_i^-(z_{\tau(3)})\\
  \cdot&
  \big[(z_{\tau(1)}-q^{-2}w)x_i^+(z_{\tau(1)})x_j^+(w)
  -(q^{-2}z_{\tau(1)}-w)
  x_j^+(w)x_i^+(z_{\tau(1)})
  , x_i^+(z_{\tau(2)})\big]
  \ot \wt x_i^+(z_{\tau(3)}).
\end{split}\end{equation}
Then it follows from (Q8) that $\bar\Delta_u(J_{ij}^+)=0$.
The proof of $\bar\Delta_u(J_{ij}^-)=0$ is similar.
Therefore, we complete the proof of Theorem \ref{thm:bialg-struct}.

\begin{rem}{\em Let $\wt{\U}_\hbar(\dot\g_{tor})$ be the quotient algebra of $\wt{\U}$ modulo the relation \eqref{strongerrela}. 
It was shown in \cite[Proposition 29]{He-representation-coprod-proof} that the action (Co1-Co4) defines a unique 
algebra homomorphism $\wt\Delta_u:\wt\U_\hbar(\dot\g_{tor})\rightarrow \big(\wt\U_\hbar(\dot\g_{tor})^{\wh\ot^2}\big) ((u))$.
Recall that $\U_\hbar(\dot\g_{tor})$ is the quotient algebra of $\wt{\U}_\hbar(\dot\g_{tor})$ modulo the relation (Q9).
 As the relations (Q6-Q8) also hold in  $\U_\hbar(\dot\g_{tor})$,
it follows from (the proof of) Theorem \ref{thm:bialg-struct}  that $\wt\Delta_u$ is compatible with the affine quantum Serre relation (Q9).
This implies that $\wt\Delta_u$ induces an algebra homomorphism from
$\U_\hbar(\dot\g_{tor})$ to $\big(\U_\hbar(\dot\g_{tor})^{\wh\ot^2}\big) ((u))$, as expected in \cite[Remark 6]{He-representation-coprod-proof}.
}\end{rem}

\section{Vertex representation of $\U$}

In this section we show the  vertex representation for $\U_\hbar(\hat\g)$ given in \cite{J-KM} induces a representation for $\U$.

We first recall the vertex representation given in \cite{J-KM}.
Let
\begin{align}
  \symalg=\C[h_{i,n}\mid i=0,1,n<0][[\hbar]]
\end{align}
be the symmetric $\C[[\hbar]]$-algebra in the variable $h_{i,n}$, which is topologically free as a $\C[[\hbar]]$-module.
It is known that there is an $\mathcal H$-action on $\symalg$ such that  $c=1$, $h_{i,0}=0$, $h_{i,-n}=$ multiplication operator and
$h_{i,n}=$ annihilation operator subject to the relation
\begin{eqnarray*}
[h_{i,m},h_{j,-n}]=\frac{\delta_{m+n,0}}{n}[n a_{ij}]_q\frac{q^{nc}-q^{-nc}}{q-q\inv},
\end{eqnarray*}
where $i=0,1$ and $n>0$.
From now on, we take $c$ and $h_{i,n}$ ( $i=0,1$, $n\in\Z$)
as operators on $\symalg$.
For $i=0,1$, we introduce the following fields on $\symalg$
\begin{align*}
&h_i^\pm(z)=\sum\limits_{\pm m>0}h_{i,m}z^{-m},\quad
E_i^\pm(z)=\te{exp}\left(\mp\sum\limits_{m>0}
    \frac{h_{i,\pm m}}{[m]_q}z^{\mp m}\right).
\end{align*}

Let $Q=\Z\al_0\oplus\Z\al_1$ be a rank $2$  lattice equipped with a semi-positive form $\<\cdot,\cdot\>$ determined by
$\<\al_i,\al_j\>=a_{ij}$ for $i,j=0,1$.
We fix a $2$-cocycle $\varepsilon:Q\times Q\to \{\pm 1\}$ such that
\begin{align*}
  &\varepsilon(\al,\be)\varepsilon(\be,\al)\inv=C(\al,\be),\quad
  \varepsilon(\al,0)=1=\varepsilon(0,\al),
\end{align*}
where
\begin{eqnarray*}\label{eq:group-commutator}
C: Q\times Q\rightarrow \{\pm 1\},\quad (\al,\be)\mapsto (-1)^{\<\al|\be\>}.
\end{eqnarray*}
Denote by $\C_\varepsilon[Q]$ the $\varepsilon$-twisted group $\C$-algebra of $Q$,
which by definition has a designated basis $\{e_\al\mid \al\in Q\}$
such that $e_\al\cdot e_\be=\varepsilon(\al,\be)e_{\al+\be}$ for $\al,\be\in Q$.
Then $\C_\hbar\{Q\}=\C_\varepsilon [Q][[\hbar]]$ becomes an $\mathcal H$-module under the action that
 $c$ and $h_{i,n}$ act trivially for $i=0,1$, $n\ne 0$, and $h_{i,0}$ ($i=0,1$) act by
\begin{align*}
  h_{i,0}\cdot e_\be=\<\be,\al_i\> e_\be\quad\te{for }\be\in Q.
\end{align*}
For $i=0,1$, define a linear operator $z^{h_{i,0}}:\C_\hbar\{Q\}\to\C_\hbar\{Q\}[[z,z\inv]]$ by
\begin{align*}
  z^{h_{i,0}}\cdot e_\be=z^{\<\beta,\al_i\>}e_\be\quad\te{for }\be\in Q.
\end{align*}

We define the Fock space
\begin{align}
  V=\symalg\wh\ot\C_\hbar\{Q\}
\end{align}
to be the tensor product  $\mathcal H$-module.
Recall the vertex operators define in \cite{J-KM}:
\begin{align}
  &X_i^\pm(z)=\sum_{n\in\Z}X_{i,n}z^{-n-1}=E_i^-(q^{\mp\half}z)^{\pm 1}E_i^+(q^{\pm\half}z)^{\pm 1}e_{\pm \al_i}z^{\pm h_{i}},
\end{align}
As usual, the normal ordered product $\:X_i^\pm(z)X_j^\pm(w)\;$ is defined by moving annihilation operators
$h_{i,n}$ ($i=0,1$, $n\ge 0$) to the right.
It was proved in \cite[(3.6)]{J-KM} that
\begin{align}\label{eq:locality}
  &X_i^\pm(z)X_j^\pm(w)=\:X_i^\pm(z)X_j^\pm(w)\;(z-q^{\mp 1}w)_{q^2}^{a_{ij}},\\
  &\:X_i^\pm(z)X_j^\pm(w)\;=(-1)^{a_{ij}}\:X_j^\pm(w)X_i^\pm(z)\;
  =\:X_j^\pm(w)X_i^\pm(z)\;,
\end{align}
where for $n\in\N$,
\begin{align*}
  &(1-x)_{q^2}^n=(1-q^{1-n}x)(1-q^{3-n}x)\cdots(1-q^{n-1}x),\\
  &(1-x)_{q^2}^{-n}=1/(1-x)_{q^2}^n,\quad
  (z-w)_{q^2}^{\pm n}=(1-w/z)_{q^2}^{\pm n}z^{\pm n},
\end{align*}
and $(1-x)_{q^2}^{-n}$ is understood as power series in $x$.
We have:

\begin{thm}\label{thm:vr}
There is a $\U$-module structure on the Fock space $V$  with the action given by
\[c\mapsto 1,\quad h_{i,n}\mapsto h_{i,n}\quad\te{and}\quad x_{i,n}^\pm\mapsto X_{i,n}^\pm\]
where  $i=0,1$ and $n\in\Z$.
\end{thm}

\begin{proof}
The relations (Q1-Q7) and (Q9) have been checked in  \cite[Theorem 3.1]{J-KM} and it remains to prove
 the relation (Q8).
For $i\ne j\in \{0,1\}$,
it follows from  \eqref{eq:locality} that
\begin{align}
  X_{ij}^\pm(z,w)=
  &(z-q^{\mp 2}w)X_i^\pm(z)X_j^\pm(w)-(q^{\mp 2}z-w)X_j^\pm(w)X_i^\pm(z)
  \nonumber\\
  =&\:X_i^\pm(z)X_j^\pm(w)\;z\inv\delta\(\frac{w}{z}\)
  =\:X_i^\pm(w)X_j^\pm(w)\;z\inv\delta\(\frac{w}{z}\).\label{eq:Q9-temp1}
\end{align}
It is straightforward to check that
\begin{align}
  &X_i^\pm(z)\:X_i^\pm(w)X_j^\pm(w)\;
  =\:X_i^\pm(z)\:X_i^\pm(w)X_j^\pm(w)\;\;\nonumber\\
  =&\:\:X_i^\pm(w)X_j^\pm(w)\;X_i^\pm(z)\;
  =\:X_i^\pm(w)X_j^\pm(w)\;X_i^\pm(z).\label{eq:Q9-temp2}
\end{align}
Then it follows from \eqref{eq:Q9-temp1} and \eqref{eq:Q9-temp2} that (Q8) holds on $V$, as required.
\end{proof}

\begin{rem}{\em From \eqref{eq:locality}, it follows that for $i\ne j\in \{0,1\}$,
\[(z-q^{\mp 2}w)X_i^\pm(z)X_j^\pm(w)\ne (q^{\mp 2}z-w)X_j^\pm(w)X_i^\pm(z).\]
In particular,  the action given in Theorem \ref{thm:vr} can not induce a $\U_\hbar(\dot\g_{tor})$-module structure on $V$.
}\end{rem}

\bigskip

\end{document}